\documentclass[oneside,a4paper,11-pt,notitlepage]{article}
\usepackage[left=0.8in, right=0.8in, top=1.5in, bottom=1.5in]{geometry}

\usepackage[T1]{fontenc} % codifica dei font
\usepackage[utf8]{inputenc} % lettere accentate da tastiera
\usepackage[english]{babel}
\usepackage{lipsum} % genera testo fittizio
\usepackage{lmodern}
\usepackage{amssymb}
\usepackage{amsthm}
\usepackage{bm}
\usepackage{mathtools}
\usepackage{braket}
\usepackage{esint}
\newcommand{\abs}[1]{{\left|#1\right|}}

\usepackage{enumitem}
\usepackage{booktabs}
\usepackage{graphicx}
\usepackage{tikz}
\usetikzlibrary{patterns}
\usepackage{multicol}
\usepackage{caption}
\usepackage{varwidth}
\usepackage{lipsum}% for context

\usepackage[skins,theorems]{tcolorbox}
\tcbset{highlight math style={enhanced,
		colframe=black,colback=white,arc=0pt,boxrule=1pt}}
\def\XXint#1#2#3{{\setbox0=\hbox{$#1{#2#3}{\int}$}
    \vcenter{\hbox{$#2#3$}}\kern-.5\wd0}}

\theoremstyle{definition}

\theoremstyle{plain}
\newtheorem{teorema}{Theorem}[section]

\newtheorem{corollario}[teorema]{Corollary}
\theoremstyle{definition}
\newtheorem{esempio}{Example}[section]
\newtheorem{oss}[esempio]{Remark}

\DeclareMathOperator{\R}{\mathbb{R}}

\makeatletter
\newcommand{\myfootnote}[2]{\begingroup
	\def\@makefnmark{}%
	\addtocounter{footnote}{-1}%
	\footnote{\textbf{#1} #2}
	\endgroup}
\makeatother

\usepackage[style = numeric]{biblatex}
\AtEveryBibitem{
    \clearfield{url}

}

\usepackage{hyperref}
\hypersetup{linktoc=none, bookmarksnumbered, colorlinks=true, linkcolor=red}

 \title{ A remark on solutions to semilinear equations with Robin boundary conditions}
\author{Celentano Antonio, Masiello Alba Lia, Paoli Gloria*}
\date{}

\newcommand{\Addresses}{{% additional braces for segregating \footnotesize 
  \bigskip 
  \footnotesize 
 
  \medskip 
 
  \textit{E-mail address}, A.~Celentano: \texttt{antonio.celentano2@unina.it} 
 
   \medskip 
 
  \textit{E-mail address}, A.L.~Masiello : \texttt{albalia.masiello@unina.it} 
   \medskip 
 
   \medskip 
     
    \textsc{Dipartimento di Matematica e Applicazioni ``R. Caccioppoli'', Universit\`a degli studi di Napoli Federico II, Via Cintia, Complesso Universitario Monte S. Angelo, 80126 Napoli, Italy.}
    
    \medskip 
       \textit{E-mail address}, G.~Paoli* (corresponding author): \texttt{gloria.paoli@fau.de} 
   
 \medskip
 
 \textsc{ Department Mathematik,
Chair in Dynamics, Control, Machine Learning and Numerics (Alexander von Humboldt-Professorship),
Cauerstr. 11,
91058 Erlangen, Germany.}
    
    \par\nopagebreak 

}}

\begin{document}
\maketitle

\begin{abstract} 
Symmetry properties of solutions to elliptic quasilinear equations have been widely studied in the context of Dirichlet boundary conditions. 
We show that, in the context of Robin boundary conditions, the symmetry property á la Gidas, Ni and  Nirenberg  does not hold in dimension $n\geq 2$, even for superharmonic functions, and we provide an explicit example. 



\end{abstract}

\section{Introduction}
The task of proving symmetries of solutions to quasilinear or nonlinear PDEs that reflect the symmetries of the domain has interested many authors. In this context, the classical result by Gidas, Ni and Nirenberg, contained in the celebrated paper \cite[Theorem 1]{GNN}, is stated as follows.

\begin{teorema}[Gidas-Ni-Nirenberg symmetry result]\label{gidas}
Let $f:\R\to\R$ be such that $f=f_1+f_2$, where  $f_1$ is a locally Lipschitz function and $f_2$ is non decreasing. Then, any positive solution $u\in C^2(\overline{B})$ to the problem
    \begin{equation}\label{ni}
\begin{cases}
    -\Delta u= f(u) & \text{in } B,\\
    u=0 &\text{on } \partial  B,
\end{cases}
\end{equation}
where $B$ is a ball of $\R^n$ of radius $R$, has to be radial and 
\begin{equation*}
\dfrac{\partial u}{\partial r}<0,\quad\text{for }\,0<r<R.
\end{equation*}
\end{teorema}

In order to prove this result the authors make use of the method of moving planes, first introduced by Aleksandrov in \cite{aleksandrov}, and then applied by Serrin in \cite{Serrin} in the context of PDEs. 

In the case $n=2$,  Lions in \cite{Lions2} gives an alternative proof of Theorem \ref{gidas}, that allows  considering weaker smoothness assumptions on $f$ and $u$, under the additional  hypothesis $f\ge0$. The technique used in \cite{Lions2} relies on the Schwartz symmetrization, the isoperimetric inequality  and the Pohozaev identity. His techniques were applied and generalized in \cite{kesevan_pacella} to the context of the $n-$Laplacian, and, eventually,  in \cite{serra} to the solutions to $p-$Laplace equation for any $p$,  in any dimension $n\geq 2$.

The result contained in \cite{GNN} is a milestone in  proving symmetry results: in the linear case see, for instance,  \cite{bereniren, berenire2, fraenkel, rosset} and, 
for the symmetry results in the case of  the $p-$Laplace operator, we refer  to \cite{ badiale, DP, DS, brock}. In all the afore-mentioned papers, it is possible to prove the radiality of positive solution to \eqref{ni} either
 under the regularity hypothesis on $f$ stated in Theorem \ref{gidas} (using the moving plane method as in \cite{GNN}) or under
 the assumption $f\ge 0$ (with symmetrization techniques as in \cite{Lions2, serra, kesevan_pacella}). For a sign changing $f$, the Lipschitz continuity property cannot be relaxed to H\"older continuity, as remarked in \cite[Section 2.3]{GNN}. Indeed, in this case, the authors find a solution to \eqref{ni} that is not radially symmetric, see also \cite{brok2} for the $p-$Laplacian case. 

The aim of the present work is to study the behaviour of the solution to the Robin problem  
\begin{equation}\label{robin}
\begin{cases}
    -\Delta u= f(u) & \text{in } B\\
    \displaystyle{\frac{\partial u}{\partial \nu}+ \beta u=0} &\text{on } \partial B
\end{cases}
\end{equation}
whenever $\beta$ is a positive parameter. Up to our knowledge, in the literature, there are few results dealing with  symmetry properties of the solutions to differential equations with Robin boundary conditions. For instance,  in \cite{beret_wei}, the authors consider the following problem

\begin{equation}\label{brni}
\begin{cases}
    -\varepsilon^2\Delta u= f(u) - u & \text{in } B\\
    \displaystyle{\varepsilon \frac{\partial u}{\partial \nu}+ \beta u=0} &\text{on } \partial B
\end{cases}
\end{equation}
 where $\varepsilon>0,$ $f: \R\to\R$ is a continuous function, of the form $f=f_1-f_2$ for $t\ge 0$, with $f_1,f_2\ge 0$, %$f_1(0)=f_1'(0)=0$ satisfying the following assumptions
 satisfying some structural growth conditions (for the precise details see \cite[Section 1]{beret_wei}).
 \begin{comment}
\begin{itemize}
    \item $f'(0)=0$;
    \item for $t\ge 0$, $f=f_1-f_2$, where $f_1,f_2\ge 0$, $f_1(0)=f_1'(0)=0$;
    \item there exists $q\geq 1$ such that $f_1(t)/t^q$ is non decreasing for any $t>0$ (the monotonicity has to be strict when $q=1$) and $f_2(t)/t^q$ is non increasing for $t>0$;
    \item $f(t)=O(t^p)$ as $t\to\infty$ for some $1<p< (n+2)/(n-2)$ (for $n=2$, we require $1<p< \infty$);
    \item $F(t)=\int_0^t f(s) \, ds$ satisfies $F(t)\le \theta t f(t)$, for some $\theta \in (0,\frac{1}{2})$.
\end{itemize}
\end{comment}
Under these assumptions, it is possible to prove the existence of a positive solution to \eqref{brni} by making use of the Palais-Smale condition and the mountain pass Lemma. The authors show that there exists a $\beta_*>1$, such that,
for $\beta>\beta_*$, the least energy solution has the maximum in the center of the ball, while, for $\beta\leq \beta_*$ and $\varepsilon\to 0$, the least energy solution has the maximum near the boundary, and, consequently, it cannot be radially symmetric.   For completeness sake, we recall that
the  case $\beta=0$, i.e. the Neumann problem, and the case $\beta=+\infty$, i.e the Dirichlet problem, have  been studied, for instance, respectively,  in \cite{takagi_91,takagi2} and  \cite{dirispike}. Moreover, these results
hold in the more general case of $\Omega\subset\R^n$ bounded and smooth, as well as the one proved in \cite{beret_wei} for the Robin boundary conditions.
\begin{comment}
 We point out
\textcolor{red}{We point out that in \cite{beret_wei} the authors do not exhibit an explicit solution to  problem \eqref{brni}. 
In particular,  we exhibit a solution  to Problem \eqref{robin}: see Theorem \ref{main}.
 It follows that in the context of Robin boundary conditions this set of hypotheses on $f$ is not enough to prove the radiality.}
 \end{comment}
 
In  Theorem \ref{th1}, we show that in the one dimensional case the symmetry result for the solution to \eqref{robin} holds under the standard hypotheses of Gidas, Ni, Nirenberg and the  additional hypothesis $f\ge 0$. On the other hand, this is not the case  for  $n\geq 2$, as pointed out in   Corollary \ref{cor1}:

\begin{quote} \it
   In dimension $n\ge 2$, there exists a positive superharmonic function $\varphi$ that is a solution to \eqref{robin}  and that is not radially symmetric.
\end{quote}

So, the main novelty of our paper is that we have found a non radial solution to problem \eqref{robin} when the nonlinearity $f$ is positive and this solution is explicit (see Theorem \ref{main}).

%the afore-mentioned set of hypothesis fails, as shown in Remark \ref{cor1}.

\section{One dimensional case: the symmetry holds.}
We start by analysing the one dimensional case.  
%\textcolor{magenta}{What is the minimal regularity we should require  on  $u$ in order to apply the result by Gidas-Ni-Nirenberg?}
\begin{teorema}\label{th1}
Let $R>0$  and let $I=]-R, R[$ be the open ball of radius $R$. Let $u\in C^2([-R,R]])$ be a solution to
\begin{equation*}
\begin{cases}
-u''=f(u) & \mbox{in } I,\\ 
\dfrac{\partial u}{\partial \nu}(x)+\beta u(x)=0 & \mbox{in } x= \pm R,
\end{cases}
\end{equation*}
 where $\beta>0$.  Let us assume that $f$ satisfies the following assumptions:
 \begin{enumerate}[label=(\roman*).]
     \item $f\ge0$ in $\R$, $f$ is not identically zero in $u(I)$,
     \item $f=f_1+f_2$, where $f_1$ is locally Lipschitz in $\R$ and $f_2$ is non decreasing.
 \end{enumerate}
  Then,   $u(x)=u(-x)$ for all $x\in [-R,R]$. Moreover, 
  \begin{equation*}\label{decr}
      u'(x)<0, \quad x\in [0,R].
  \end{equation*}

\end{teorema}

\begin{proof}
We  divide the proof in two steps. In the first step, we prove that the function $u$ is strictly positive and, in the second one, we prove that we can apply the result contained in Theorem \ref{gidas}.

\textbf{Step 1}.
We start by proving that  $u>0$ in $[-R,R]$.  Since $u''\leq 0$, $u'$ is non increasing in $]-R, R[$, so the minimum of $u$ on $[-R,R]$ is achieved either in $-R$ or in $R$.
Let us denote by $x_m$ the minimum point of $u$ in $[-R, R]$. 
From the Robin boundary conditions, we have that 
\begin{equation*}
-\beta u(x_m)=\dfrac{\partial u}{\partial \nu}(x_m)\leq 0,
\end{equation*}
and, as a consequence, $u\geq 0$ in $[-R,R]$. 

Now we want to  prove that $u>0$ in $[-R,R]$. By contradiction, we assume that $u(x_m)=0$. If $x_m=-R$, the Robin boundary conditions imply
$$0=\beta u(-R)=-\frac{\partial u}{\partial \nu}(-R)= u'(-R)\ge u'(x), \quad \forall x\in I,$$
where we have observed that $u'$ is non increasing. This implies that also $u$ is a non increasing function, so  $-R$ should be both a minimum and a maximum. This is not possible, since $u$ should be constant and this contradicts the hypothesis $f\not\equiv0 $ in $u(I)$.
%Let us observe that   $x_m\neq-R$. Indeed, if $x_m=-R$, from the Robin boundary condition we obtain  that $u$ is a decreasing function, so, it is not possible that the minimum point is achieved  in $-R$.   
Therefore,  we have that  $x_m=R$ and, arguing as before, we have

$$0=-\beta u(R)=\frac{\partial u}{\partial \nu}(R)= u'(R)\le u'(x), \quad \forall x\in I,$$

So $u$ is a non decreasing function, the point $R$ is both a minimum and a maximum and we get a contradiction as before.   

\begin{comment}
If it exists a critical point $x_c\in\Omega$, then $u'\leq 0$ in $[x_c, R]$ and $u$ achieves its absolute minimum at a point $\tilde {x}$ on the boundary. Therefore 
\begin{equation}{-\beta u(\tilde{x})=\dfrac{d u}{d \nu}(\tilde{x})\leq 0,\nonumber}\end{equation}
and we have $u(\tilde{x})\geq 0$. We assume by contradiction that $u(\tilde{x})=0$ and so we have $\dfrac{d u}{d \nu}(\tilde{x})=0$. We distinguish two cases. If $\tilde{x}=-R$, then $u(R)<0$, which is a contradiction. If $\tilde{x}=R$, then $u'\geq 0$ in $\overline{\Omega}$ and $\tilde{x}$ is also a maximum point for $u$. In conclusion, $u$ is constant and we obtain a contradiction because $f$ is supposed to be not identically zero.
\end{comment}

\textbf{Step 2}. 
We prove now that  $u(R)=u(-R)$. %and, so, the thesis follows directly by applying  Theorem \ref{gidas}. 
Let us assume  by contradiction that $u(R)  \neq u(-R)$. Without loss of generality, we can suppose
$u(R)<u(-R)$. 
As a consequence of Step $1$, the function $u$ is strictly increasing in a neighborhood of the point $-R$, so, by the continuity of $u$, there exists $y\in I$ such that $u(y)=u(-R)$. Therefore, the following quantity is well defined  $$\lambda:=\inf\{t\in I\,:\,u(t)=u(-R)\}$$ 
and $\lambda>-R$.  
Moreover, the continuity of $u$  also implies $u(\lambda)=u(-R)$ and $u(x)>u(-R)$
in $(-R, \lambda)$.

%From the continuity of $u$, we have that $\lambda>-R$, $u(\lambda)=u(-R)$ and $u(x)>u(-R)$ on the interval $(-R, \lambda)$.
%Since $u(-R)>0$, from the Robin boundary condition,  it follows that $u$ is strictly increasing in a right neighbourhood of $x=-R$ and we have that  $\lambda>-R$. Moreover, from the continuity of $u$, we have $u(\lambda)=u(-R)$ and $u(x)>u(-R)$ on the interval $(-R, \lambda)$.

We define now the function  $v:=u-u(-R)$, 
that  is a positive solution to 
   \begin{equation*}\label{til}
\begin{cases}
    -\Delta v= \Tilde{f}(v) & \text{in } (-R, \lambda),\\
    v=0 &\text{in } x=-R, \; x=\lambda,
\end{cases}
\end{equation*}
where $\tilde{f}(v)= f(v+u(R))$.
So we can use   Theorem \ref{gidas} in the interval $(-R, \lambda)$. We have that $u$ is symmetric  with respect to the line $x=2^{-1}(\lambda-R)$ and, as a consequence, we get 
\begin{equation}\label{sym}
\dfrac{d u}{d \nu}(-R)=-\dfrac{d u}{d x}(-R)=\dfrac{d u}{d x}(\lambda)<0.
\end{equation}
Using  \eqref{sym} and the fact that   $u'$ is non increasing, we obtain 
\begin{equation}
\beta u(R)=-\dfrac{d u}{d \nu}(R)=-\dfrac{d u}{d x}(R)\geq -\dfrac{d u}{d x}(\lambda)=-\dfrac{d u}{d \nu}(-R)=\beta u(-R), 
\end{equation}
and, therefore,
\begin{equation}
u(R)\geq u(-R),
\end{equation}
that is a contradiction.

From Step $1$ and $2$, we have that the funtion $v=u-u(R)$  is a positive solution to 
   \begin{equation*}
\begin{cases}
    -\Delta v= \Tilde{f}(v) & \text{in } I,\\
    v=0 &\text{in } x=\pm R.
\end{cases}
\end{equation*}
So we can conclude by using  Theorem \ref{gidas}. 
%It is a contradiction, so $u(R)=u(-R)$. The statement follows by applying the Theorem 2.1 in \cite{GNN}.

\end{proof}
We do not know if the hypotheses $i)-ii)$ on the function $f$ are the  optimal ones to obtain the symmetry result. Nevertheless,  we will show in   Remark \Ref{rem1d}, in the next session, that  the assumption $f\geq 0$ cannot be removed.

\section{Counterexample:   symmetry  breaking in dimension $n\geq 2$}
In the following, we will denote by $|\cdot|$ the Euclidean norm in $\mathbb{R}^n$.

\begin{comment}
\begin{teorema}\label{main}
    Let $B_R\subset \mathbb{R}^n$, $n\geq 2$, be the ball centered at the origin with radius $R$, let $\beta$ be a positive constant and let $x_0\neq 0$ in $B_R$. Then, the positive function $\varphi\in C^\infty(\overline{B_R})$, defined as 
    \begin{equation}
    \label{timoteo}
    \varphi(x)=\dfrac{1}{(\abs{x-x_0}^2+R^2-\abs{x_0}^2)^{\beta R}}, 
    \end{equation}
     is a non radial function in $B_R$ (i.e. $\varphi(x)\neq \varphi(\abs{x})$) and it is a solution to 
     \begin{equation}\label{cont}
        \begin{cases}
            -\Delta \varphi= f(\varphi)
          &\text{in }B_R,\\
            \dfrac{\partial \varphi}{\partial \nu}+\beta \varphi=0  &\text{on }\partial B_R,
        \end{cases}
    \end{equation}
     where
\begin{equation}\label{f}
    f(t)=c_1 t\left[ c_2 t^{\frac{1}{\beta R}}+ c_3 t^{\frac{2}{\beta R}}\right],
\end{equation}
with  $c_1, c_3>0$, $c_2\in \R$   defined as follows:
    \begin{equation*}
        c_1=2\beta R, \quad c_2=-2(\beta R+1) +n, \quad c_3=2(\beta R+1)\alpha^2, \quad \alpha^2=R^2-\abs{x_0}^2.
    \end{equation*}
\end{teorema}
\end{comment}

\begin{teorema}\label{main}
    Let $B_R\subset \mathbb{R}^n$, $n\geq 2$, be the ball centered at the origin with radius $R$, let $\beta$ be a positive constant and let $x_0\neq 0$ in $B_R$. 
Then, there exists a the positive function $\varphi\in C^\infty(\overline{B_R})$ that is a non radial function in $B_R$ (i.e. $\varphi(x)\neq \varphi(\abs{x})$) and it is a solution to
     \begin{equation}\label{cont}
        \begin{cases}
            -\Delta \varphi= f(\varphi)
          &\text{in }B_R,\\
            \dfrac{\partial \varphi}{\partial \nu}+\beta \varphi=0  &\text{on }\partial B_R,
        \end{cases}
    \end{equation}
     where
\begin{equation}\label{f}
    f(t)=c_1 t\left[ c_2 t^{\frac{1}{\beta R}}+ c_3 t^{\frac{2}{\beta R}}\right],
\end{equation}
with  $c_1, c_3>0$, $c_2\in \R$   defined as follows:
    \begin{equation*}
        c_1=2\beta R, \quad c_2=-2(\beta R+1) +n, \quad c_3=2(\beta R+1)\alpha^2, \quad \alpha^2=R^2-\abs{x_0}^2.
    \end{equation*}
\end{teorema}

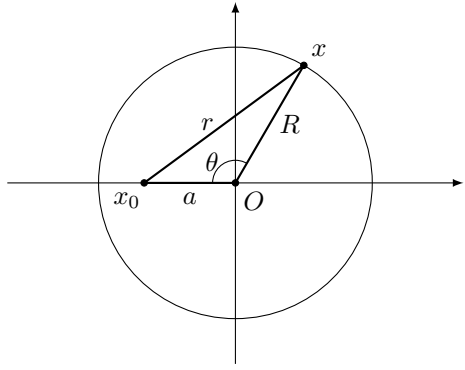
\begin{figure}[h!]
\begin{center}
    \begin{tikzpicture}[scale=1.2]
        \draw[black, -latex, thin] (-1.5,0) -- (3.5,0);
        \draw[black, -latex, thin] (1,-2) -- (1,2);

        \draw[black] (1,0) circle(1.5);
        \filldraw[black] (1,0) circle(1pt) node[anchor = 135] {$O$};
        \draw[black, thick] (0,0) -- (1,0) node[pos = 0.5, below] {$a$};
        \filldraw[black] (0,0) circle(1pt) node[anchor = 45] {$x_0$};

        \filldraw[black] (1,0) ++ (60:1.5) circle(1pt) node[anchor = -135] {$x$};

        \draw[black, thick] (1,0) --++ (60:1.5) node[pos = 0.5, right] {$R$};

        \draw[black, thick] (1,0) ++ (60:1.5) -- (0,0) node[pos = 0.5, left] {$r$};

        \begin{scope}
            \clip (0,0) -- (1,0) --++ (60:1.5) -- cycle;
           \draw[black] (0.75,0) arc(180:0:0.25) node[pos = 0.4, left] {$\theta$};
        \end{scope}
 \end{tikzpicture}
\end{center}
    \caption{Construction of the function $\varphi$} \label{fig}
    \end{figure}

\begin{proof}
 We define the following quantities (see Figure \ref{fig}):
\begin{itemize}
    \item $a:=\abs{x_0}$,
    \item $\alpha^2:=R^2-a^2>0$.
\end{itemize}
We show the existence of a positive function $\varphi(x)=\varphi(\abs{x-x_0})=\varphi(r)$ 
such that 
\begin{equation}\label{Robin_sphere}
\dfrac{\partial \varphi}{\partial \nu}+\beta \varphi=0\qquad \text{on}\;\;\partial B_R
\end{equation}
where  $\nu$  is the unit outer normal to $\partial B_R$.
Let us fix $x\in\partial B_R$. Being
\begin{equation*}
\nabla(\varphi(x))=\varphi'(r) \:\dfrac{x-x_0}{r},\qquad \nu(x)=\dfrac{x}{R},
\end{equation*}
the Robin boundary conditions \eqref{Robin_sphere} becomes
\begin{equation}\label{robin_radial}
\varphi'(r)\;\dfrac{x-x_0}{r}\cdot \dfrac{x}{R}+\beta\varphi (r)=0.
\end{equation}
Denoting, now,  by $\theta$ the angle between the vectors $x_0$ and $x$, we have the following relation
\begin{equation}\label{cos}
   \cos(\theta)=\dfrac{R^2+a^2-r^2}{2aR}. 
\end{equation}
So, from \eqref{robin_radial} and \eqref{cos}, recalling that $\alpha^2=R^2-a^2$, we have
\begin{equation*}
\frac{\varphi'(r)}{r R} \left( R^2-R a \cos(\theta)\right)+\beta \varphi(r)=\frac{\varphi'(r)}{ 2 r R} \left( r^2+\alpha^2\right) +\beta \varphi(r)=0,  \end{equation*}
and, therefore, 
\begin{equation} \label{final_eq}
\dfrac{\varphi'(r)}{\varphi(r)}=-\left( 2\beta R\right)\dfrac{r}{r^2+\alpha^2}.
%{\varphi'(r) \left(\dfrac{r^2+\alpha^2}2rR \right)+\beta\varphi=0.}
\end{equation}
Integrating \eqref{final_eq}, we get 

\begin{equation}\label{timoteo}
    \varphi(r)=\dfrac{c}{(r^2+\alpha^2)^{\beta R}}.
\end{equation}
If we choose $c=1$, we have 
\begin{align}
\varphi'(r)&=-\dfrac{2 \beta R r}{(r^2+\alpha^2)^{\beta R+1}}\nonumber\\
\varphi''(r)&=-(-\beta R-1)\dfrac{4\beta R r^2}{(r^2+\alpha^2)^{\beta R+2}}-\dfrac{2 \beta R}{(r^2+\alpha^2)^{\beta R+1}}\nonumber.
\end{align}
and, consequently,
\begin{align*}
    -\Delta \varphi&=-\varphi''(r)-\dfrac{n-1}{r}\varphi'(r)=\\
    &=\dfrac{2\beta R}{(r^2+\alpha^2)^{\beta R+1}}\left[n-2(\beta R+1)+(\beta R+1)\dfrac{2\alpha^2}{r^2+\alpha^2}\right]\\
&=2\beta R\varphi(r)^{\frac{1}{\beta R}+1}\left[n-2(\beta R+1)+ 2(\beta R+1)\alpha^2\varphi(r)^{\frac{1}{\beta R}}\right]\\
&=f(\varphi(r)),
\end{align*}
where $f$ is the function defined in \eqref{f}.
So, we have proved the desired claim, since we have found a non radial function of the form $\varphi(x)=\varphi(\abs{x-x_0})=\varphi(r)$, defined in \eqref{timoteo}, that satisfies \eqref{cont}.
\end{proof}

As a  consequence of Theorem \ref{main},  we obtain the following Corollary. 

\begin{corollario}\label{cor1}
Let $n\ge 2$. There exists a positive superharmonic function $\varphi$ that is a solution to \eqref{robin}  and that is not radially symmetric.
\end{corollario}

\begin{proof}
    In the case $n=2$, the right-hand side of  \eqref{f} becomes 
$$ 
f(t)=4\beta R t \left(-\beta R t^{\frac{1}{\beta R}}+\alpha^2(1+\beta R) t^{\frac{2}{\beta R}}\right).
$$
We notice that
\begin{equation}
    \label{bound}
  f(t)\geq0,\qquad {\rm if }  \;t\geq\left(\dfrac{\beta R}{\alpha^2(1+\beta R)}\right)^{\beta R},
\end{equation}
 so the function $f\circ\varphi$ is positive if $\varphi$ satisfies \eqref{bound} for all $x\in B_R$,
and this  follows by imposing the following geometric constraint
$$\beta\leq\dfrac{R-a}{R(R+a)}.$$
If $n\geq 3$, we can choose the constant $c_2\geq 0$, by imposing the condition
\begin{equation}
    \label{n3}
    \beta\leq \dfrac{n-2}{2R}
\end{equation}
and, under these assumptions, we have that $f(t)\geq 0$ for $t\geq 0$.

Therefore, we can see that, by imposing the geometrical constraints \eqref{bound} and \eqref{n3} respectively for   $n=2$ and $n\geq 3$, the function $\varphi$ defined in \eqref{timoteo} is an example of positive superharmonic function, that is non radial and that satisfies \eqref{cont}.
\end{proof}

%Some remarks are now in order to put the result contained in Theorem \ref{main} in the right framework.

We  conclude  with a remark  for the one dimensional case.

\begin{oss}\label{rem1d}
The function $\varphi$ defined in Theorem \ref{main},  in the case $n=1$,  satisfies the problem 
 \begin{equation}
        \begin{cases}
            - \varphi''= f(\varphi)
          &\text{in } (-R,+R),\\
            \dfrac{\partial \varphi}{\partial \nu}+\beta \varphi=0  &\text{ in} \;x=\pm R.
        \end{cases}
    \end{equation}
  We note that $\varphi\in C^\infty \left([-R, R]\right)$ and $f$ is a locally Lipschitz function, but, $f$ does not satisfy the hypothesis $i)$, that is the positiveness.  
Indeed, by straightforward computations, we obtain that $f\circ\varphi$ is a sign changing function  in  $\varphi([-R, +R])$ for every $\beta>0$ and $R>a>0$. 
\end{oss}

%We conclude with the following open problem.

% Finally we conclude with a list of open problems and possible future developments.

%\section{Conclusions and possible developments}

%\section*{Acknowledgements}
%This work has been partially supported by GNAMPA of INdAM. In particular, the author Gloria Paoli is supported by the Alexander von Humboldt Foundation with an Alexander von Humboldt research fellowship.  Moreover, the authors would like to thank the reviewers for their suggestions to improve
%this paper.

\section*{Acknowledgement and Declarations}

The authors would like to thank Professor Carlo Nitsch for the useful and insightful discussions.

\vspace{2mm}
\noindent \textbf{Funding:} All the  authors are  partially supported by Gruppo Nazionale per l’Analisi
Matematica, la Probabilità e le loro Applicazioni (GNAMPA) of Istituto Nazionale di Alta
Matematica (INdAM). Gloria Paoli is    supported by the Alexander von Humboldt Foundation through an Alexander von Humboldt research fellowship.

\noindent \textbf{Conflict of interst:} There is no conflict of interest to disclose

%\section*{Acknowledgment}

%\bibliographystyle{plain}
%\bibliography{biblio}

\printbibliography

\Addresses 

\end{document}